\documentclass[]{article}
\usepackage{amssymb}
\usepackage{amsthm}
\usepackage{amsmath,amsfonts,latexsym}
\usepackage{graphicx}
\usepackage{cite}

\usepackage{amsfonts,latexsym}
\usepackage[colorlinks=true]{hyperref}
\usepackage{url,color}

\newtheorem{theorem}{Theorem}

\newtheorem{proposition}[theorem]{Proposition}
\newtheorem{lemma}[theorem]{Lemma}

\theoremstyle{definition}

\newtheorem{definition}[theorem]{Definition}

\theoremstyle{remark}
\newtheorem{remark}[theorem]{Remark}
\newcommand{\R}{\mathbb{R}}
\newcommand{\N}{\mathbb{N}}

\newcommand{\F}{\mathcal{F}}

\begin{document}
%
\title{Fractional Discrete Systems with Sequential $h$-differences}

\author{ Ewa Girejko, Dorota Mozyrska, Ma{\l}gorzata Wyrwas}
\date{Faculty of Computer Science\\
        Department of Mathematics\\
        Bia{\l}ystok University of Technology \\
        Wiejska 45A,
        15-351 Bia\l ystok, Poland\\[2mm]
        \texttt{$\{$e.girejko, d.mozyrska, m.wyrwas$\}$@pb.edu.pl}}
\maketitle

\begin{abstract}
In the paper we study the subject of positivity of systems with sequential fractional difference. We give formulas for the unique solutions to systems in linear and semi-linear cases. The positivity of systems is considered.
\end{abstract}

\section{Introduction}
The first definition of the fractional derivative was introduced by Liouville and Riemann at the end of the 19-th century. Later on, in the late 1960s, this idea was used by engineers for modeling various processes. Thus the fractional calculus started to be exploited since that time.  This calculus is a field of mathematics that grows out of the traditional definitions of calculus integral and derivative operators and deals with fractional derivatives and integrals of any order. Fractional difference calculus have been investigated by many authors, for example, \cite{BalAbd,AticiEloe,atici1,ChenLou,Holm,GirMoz,Kaczorek2008,MilRos,DM_EP12,ORT,Pod} and others. The subject of positivity is well developed for fractional linear systems with continuous time, see \cite{Kaczorek1,Kaczorek2008,Kaczorek2009}. However, positivity of fractional discrete systems with sequential $h$-differences is still a field to be examined. In the present paper we open our studies in this field. We give formulas for the unique solutions to systems in linear and semi-linear cases. Moreover, the 
positivity of systems is considered.\\
The paper is organized as follows. In Section \ref{sec:2} all preliminaries definitions, facts and notations are gathered. Section \ref{sec:3} presents systems with sequential fractional differences with results on uniqueness of solutions. Semilinear systems we included in Section \ref{sec4}. The last Section concerns positivity of considered systems.

\section{Preliminaries}\label{sec:2}

Let us denote by $\mathcal{F}_D$ the set of real valued functions defined on $D$.
Let $h>0, \alpha>0$ and put $(h \mathbb{N})_a:=\{a,a+h,a+2h,...\}$ for $h>0$ and $a\in\R$. Let
\begin{equation*} R^n_\geq=\{x\in\mathbb{R}^n: x_i\geq 0, 1\leq i\leq n\}.\end{equation*}
Due to   notations from the time scale theory  the operator $\sigma:(h \mathbb{N})_a\to (h \mathbb{N})_a$ is defined by $\sigma(t):=t+h$. 
The next definitions of $h$-difference operator was originally given in \cite{BasFerTor}, but here we propose simpler notation.
\begin{definition}\label{def:1}
For a function $x\in\F_{(h\mathbb{N})_a}$ the {\em forward $h$-difference operator} is defined as
\begin{equation*}
(\Delta_h x)(t):=\frac{x(\sigma(t))-x(t)}{h},\ \ \,t=a+nh, \, n\in \N_0\,,
\end{equation*}
while the {\em $h$-difference sum} is given by
\begin{equation*}
\left(_a\Delta^{-1}_hx\right)(t):=h\sum^{n}_{k=0}x(a+kh)\,,
\end{equation*}
where $t=a+(n+1)h$, $n\in\N_0$ and $\left(_a\Delta^{-1}_hx\right)(a)=0$.
\end{definition}

\begin{definition}\cite{BasFerTor}\label{def:2}
For arbitrary $t,\alpha\in\R$ the {\em $h$-factorial function} is defined by
\begin{equation}\label{def:h-factorial_function}
t^{(\alpha)}_h:=h^\alpha\frac{\Gamma(\frac th+1)}{\Gamma(\frac th+1-\alpha)}\,,
\end{equation}
where $\Gamma$ is the Euler gamma function, $\frac{t}{h}\not\in \mathbb{Z}_{-}:=\{-1,-2,-3,\ldots\}$, and we use the convention that division at a pole yields zero.
\end{definition}
Notice that if we use the general binomial coefficient $\binom{a}{b}:=\frac{\Gamma(a+1)}{\Gamma(b+1)\Gamma(a-b+1)}$, then \eqref{def:h-factorial_function} can be rewritten as
\[t^{(\alpha)}_h=h^\alpha\Gamma(\alpha+1)\binom{\frac{t}{h}}{\alpha}\,.\]

The next definition with another notations was stated in \cite{BasFerTor}. Here we use more suitable summations.
\begin{definition}\label{sum}
For a function $x\in\F_{(h\mathbb{N})_a}$ the \emph{fractional $h$-sum of order $\alpha>0$} is given by
\begin{equation*}
\left(_a\Delta^{-\alpha}_hx\right)(t):=\frac{h}{\Gamma(\alpha)} \sum^{n}_{k=0}(t-\sigma(a+kh))^{(\alpha-1)}_hx(a+kh)\,,
\end{equation*}
where $t=a+(\alpha+n)h$, $n\in \mathbb{N}_0$. 
Moreover we define $\left(_a\Delta^0_hx\right)(t):=x(t)$.
\end{definition}
\begin{remark}
Note that $_a\Delta^{-\alpha}_h: \F_{(h\mathbb{N})_a}\rightarrow \F_{(h\mathbb{N})_{a+\alpha h}}$.
\end{remark}
Accordingly to the definition of $h$-factorial function the formula given in Definition \ref{sum} can be rewritten as:
\[\begin{split}
          \left(_a\Delta^{-\alpha}_h
x
\right)(t) & = h^{\alpha}\sum^{n}_{k=0}\frac{\Gamma(\alpha+n-k)}{\Gamma(\alpha)\Gamma(n-k+1)}
x
(a+kh) \\
            & =
h^{\alpha}\sum^{n}_{k=0}\binom{n-k+\alpha-1}{n-k}
x
(a+kh)\,
\end{split}
\]
for $t=a+(\alpha+n)h$, $n\in \mathbb{N}_0$.

\begin{remark}
In \cite{Holm} one can find the following form of the fractional $h$-sum of order $\alpha>0$:
\begin{equation*}\label{sum_operator:holm}
\left(_a\Delta^{-\alpha}_hx\right)(t)=\frac{h^\alpha}{\Gamma(\alpha)} \sum^{t-\alpha h}_{k=a}\left(\frac{t-\sigma(k)}{h}\right)^{(\alpha-1)}x(k)
\end{equation*}
that can be useful in implementation.
\end{remark}

The following definition one can find in \cite{Abd} for $h=1$.

\begin{definition}\label{DEF:Caputo:h-difference:operator}
Let $\alpha\in(0,1]$. The \emph{Caputo $h$-difference operator} $_a\Delta_{h,*}^{\alpha}x$ of order $\alpha$ for a function $x\in\F_{(h\mathbb{N})_a}$ is defined by
\begin{equation*}\label{operator:C}
\left(_a\Delta_{h,*}^{\alpha}x\right)(t):=\left(_a\Delta^{-(1-\alpha)}_h\left({\Delta}_hx\right)\right)(t),\ \ \,
t\in (h\mathbb{N})_{a+(1-\alpha) h}\,.
\end{equation*}
\end{definition}
\begin{remark}
Note that: $_a\Delta_{h,*}^{\alpha}: \F_{(h\mathbb{N})_a}\rightarrow \F_{(h\mathbb{N})_{a+(1-\alpha) h}}$, where $\alpha\in(0,1]$.
\end{remark}

We need the power rule formulas in the sequel.
Firstly, we easily notice that for $p\neq 0$
the well defined $h$-factorial functions have the following property:
\begin{equation*}\label{pochodna}
\Delta_h(t-a)_h^{(p)}=p(t-a)_h^{(p-1)}\,.
\end{equation*}
More properties of $h$-factorial functions can be found in \cite{GirMozCoimbra}.
In our consideration the crucial role plays the power rule formula presented in \cite{FerTor}, i.e.
\begin{equation}\label{power:rule:FerTor}
\left(_a\Delta_h^{-\alpha}\psi\right)(t)=\frac{\Gamma(\mu+1)}{\Gamma(\mu+\alpha+1)}\left(t-a+\mu h \right)_h^{(\mu+\alpha)}\,,
\end{equation}
where $\psi(r)=(r-a+\mu h)_h^{(\mu)}$, $r\in (h\N)_a$, $t\in (h\N)_{a+\alpha h}$.
Note that using  the general binomial coefficient one can write \eqref{power:rule:FerTor} as
\[\left(_a\Delta_h^{-\alpha}\psi\right)(t)=\Gamma(\mu+1)\binom{n+\alpha+\mu}{n}h^{\mu+\alpha}\,.\]
Then if $\psi\equiv 1$, then we have for
$\mu=0$,
$a=(1-\alpha)h$ and $t=nh+a+\alpha h$
\begin{equation*}\begin{split}
\left(_{a}\Delta^{-\alpha}_h 1 \right) (t) &  =\frac{1}{\Gamma(\alpha+1)}(t-a)_h^{(\alpha)}
\\ & = \frac{\Gamma(n+\alpha+1)}{\Gamma(\alpha+1)\Gamma(n+1)}h^{\alpha}  =\binom{n+\alpha}{n}h^\alpha\,.
\end{split}
\end{equation*}

Let us define special functions, that we use in the next  section to write the formula for solutions.
\begin{definition}\label{fi_def}
For $\alpha, \beta >0$ we define
\begin{equation}\label{fi_def1}
\varphi_{k,s}(nh):=\left\{\begin{array}{ll}\binom{n- k+k\alpha+s\beta}{n-k} h^{k\alpha +s\beta}, & \mbox{for} \ n\in\N_{k}\\
0, & \mbox{for} \ n<k\end{array}\right.
\end{equation}
\end{definition}
and
\begin{equation}\label{fi_def2}
\widetilde{\varphi}_{k,s}(nh):=\left\{\begin{array}{ll}
\binom{n+\mu-1}{n}h^{\mu}=\frac{\Gamma(n+\mu)}{\Gamma(\mu)\Gamma(n+1)}h^{\mu}\,, & \mbox{for} \ n\in\N_0\\
0\,, & \mbox{for} \ n<0\end{array}\right.\,,
\end{equation}
where $n,k,s\in\mathbb{N}_0$ and $\mu=k\alpha+s\beta$.

\begin{remark}\label{rem}
It is worthy to notice that:
 \begin{itemize}
 \item[(a)] $\varphi_{0,0}(nh)=1$;
  \item[(b)] $\varphi_{1,0}(nh)=\binom{n+\alpha-1}{n-1}h^{\alpha} = \left(_0\Delta_h^{-\alpha}1\right)\left((n-1)h+\alpha h\right)$;
   \item[(c)] $\varphi_{k,s}(nh)=\frac{\Gamma(n-k+1+k\alpha+s\beta)}{\Gamma(k\alpha+s\beta+1)\Gamma(n-k+1)}$ and as the division by pole gives zero, the formula works also for $n<k, n\in\N$;
 \item[(d)]  $\varphi_{k,s}(nh)=\frac{1}{\Gamma(k\alpha+s\beta+1)}\left((n-k)h+k\alpha h+s\beta h\right)^{(k\alpha+s\beta)}_h$.
 \end{itemize}
\end{remark}

We also need the property presented in the following proposition.

\begin{proposition}\label{fi}
Let $\alpha, \beta \in (0,1], h>0$ and $a=(\alpha-1)h, b=(\beta-1)h$.
Then
\begin{equation}\label{powerrule}
\left(_0\Delta^{-\alpha}_h \varphi_{k,s}\right)(nh+a)=\varphi_{k+1,s}(nh)
\end{equation}
and
\begin{equation}\label{powerrule_f}
\left(_0\Delta^{-\beta}_h \varphi_{k,s}\right)(nh+b)=\varphi_{k,s+1}(nh)\,.
\end{equation}
\end{proposition}
\begin{proof}
We show only equality (\ref{powerrule}), as 
\eqref{powerrule_f} is a symmetric one.\\
Let $\mu:=k\alpha+s\beta$. For $r\in (h\N)_{kh}$ we define the following $h$-factorial function
$\psi(r):=(r+\mu h)_h^{(\mu)}$.
Since
\[
\begin{split}
  \varphi_{k,s}(nh) & = \frac{1}{\Gamma(k\alpha +s\beta +1)}\left((n-k)h+k\alpha h + s\beta h\right)_h^{(k\alpha+s\beta)}\\
    & =
\frac{1}{\Gamma(\mu +1)}\psi(nh-kh)\,
\end{split}
\]
for $n\geq k$ and $\varphi_{k,s}(mh)=0$ for $m<k$, by \eqref{power:rule:FerTor}
we get
\begin{eqnarray*}\begin{split}
 \left(_0\Delta_h^{-\alpha}\varphi_{k,s}\right)(t) = \left(_{kh}\Delta_h^{-\alpha}\varphi_{k,s}\right)(t)
 & =\frac{1}{\Gamma(\mu +1)}\left(_0\Delta_h^{-\alpha}\psi\right)(t) \\ & =
 \frac{1}{\Gamma(\mu+1)}\frac{\Gamma(\mu+1)}{\Gamma(\mu+\alpha+1)}\left(t+\mu h \right)_h^{(\mu+\alpha)}\\
 &=\frac{1}{\Gamma(\mu+\alpha+1)}\left(t+\mu h \right)_h^{(\mu+\alpha)}\,,
 \end{split}
\end{eqnarray*}
where 
$t=a-kh+nh$.
Hence
\[t+\mu h=nh-(k+1)h +(k+1)\alpha h +s\beta h\]
and
\begin{eqnarray*}\begin{split}
 \left(_0\Delta_h^{-\alpha}\varphi_{k,s}\right)(nh+a) & =\frac{1}{\Gamma(\mu+\alpha+1)}\left(nh+a+\mu h \right)_h^{(\mu+\alpha)}\\& =
 \frac{\Gamma(\alpha +n-(k+1)+\mu +1)}{\Gamma(\mu+\alpha+1)\Gamma(n-(k+1)+1)}h^{\mu+\alpha}\\
 &=
 \frac{\Gamma(\alpha +n-k+\mu)}{\Gamma(\mu+\alpha+1)\Gamma(n-k)}h^{\mu+\alpha}\\
 &=
 \binom{n-k-1+\mu+\alpha}{n-k-1}h^{\mu+\alpha}\\
 &=\binom{n-(k+1)+(k+1)\alpha+s\beta}{n-(k+1)}h^{(k+1)\alpha+s\beta}\\
 &
 =\varphi_{k+1,s}(nh).
\end{split}
\end{eqnarray*}
\end{proof}

From the application of the power rule  follows the rule for composing two fractional $h$-sums. The proof for the case $h=1$ one can find in \cite{Holm}. For any positive $h>0$ we presented the prove in \cite{GirMozCoimbra}.

\begin{proposition}\label{sumy}
Let $x$ be a real valued function defined on $\left(h\N\right)_a$, where $a, h\in\R, h>0$. For $\alpha, \beta>0$ the following equalities hold:
\begin{equation*}\label{eq:sumy}
\begin{split}
\left( _{a+\beta h}\Delta_h^{-\alpha}\left( _{a}\Delta_h^{-\beta}x\right)\right)(t) & =\left( _{a}\Delta_h^{-\left(\alpha+\beta\right)}x\right)(t)\\ & =\left( _{a+\alpha h}\Delta_h^{-\beta}\left( _{a}\Delta_h^{-\alpha}x\right)\right)(t)\,,
\end{split}
\end{equation*}
where $t\in\left(h\N\right)_{a+\left(\alpha+\beta\right)h}$.
\end{proposition}

The next proposition gives a useful identity of transforming Caputo fractional difference equations into fractional summations for the case
when an order is from the interval $(0,1]$.

\begin{proposition}\label{lemat12}\cite{GirMozCoimbra}
Let $\alpha\in(0,1]$, $h>0$, $a=(\alpha-1)h$ and $x$ be a real valued function defined on $\left(h\N\right)_a$.
 The following formula holds
\begin{equation*}
\left(_0\Delta_{h}^{-\alpha} \left(_{a}\Delta^\alpha_{h,*}x\right)\right)(nh+a)=x(nh+a)- x(a),\ \ \, n\in\N_1\,.
\end{equation*}
\end{proposition}

\section{Systems with sequential  fractional differences}\label{sec:3}
Let $\alpha, \beta\in (0,1]$ and $x:\left(h\N\right)_a\rightarrow \R^n$. Moreover, let us take
$a=(\alpha-1)h$ and $b=(\beta-1)h$.
Then we define
\begin{equation*}
y(nh+b):=\left(_a\Delta^{\alpha}_{h,*}x\right)(nh)\,.
\end{equation*}
Note that $y:\left(h\N\right)_b\rightarrow \R^n$.
Then we apply the next difference operator of order $\beta$ on the new function $y$ and consider here an initial value problem stated by the system:
\begin{subequations}\label{uklad:ogolny}
\begin{eqnarray}
  \left(_a\Delta^{\alpha}_{h,*}x\right)(nh) &=&  y(nh+b)\,,  \\
   \left(_b\Delta^{\beta}_{h,*}y\right)(nh) & =& f(nh, x(nh+a))
\end{eqnarray}
\end{subequations}
with initial values:
\begin{subequations}\label{warunki}
\begin{eqnarray}
  \left(_a\Delta^{\alpha}_{h,*}x\right)(0)  & =& x_0\,,\\
   x(a) & =& x_a\,,
\end{eqnarray}
\end{subequations}
where $x_a, x_0$ are constant vectors from $\R^n$.
We use Proposition \ref{lemat12} twice and, for $n\geq1$, we get:
\begin{equation*}\begin{split}
\left(_0\Delta_h^{-\beta}\left( _b\Delta^{\beta}_{h,*}y\right)\right)(nh+b) & =y(nh+b)-y(b) \\ & =\left( _a\Delta^{\alpha}_{h,*}x\right)(nh)-x_0
\end{split}
\end{equation*}
and
\begin{equation*}
\left(_0\Delta_h^{-\alpha}\left( _a\Delta^{\alpha}_{h,*}x\right)\right)(nh+a)=x(nh+a)-x_a\,.
\end{equation*}
Hence
\begin{equation}
\left( _a\Delta^{\alpha}_{h,*}x\right)(nh)=x_0+\left(_0\Delta^{-\beta}_h\tilde{f}\right)(nh+b)\,,
\end{equation}
where $\tilde{f}(nh):=f(nh, x(nh+a))$.
Nextly
\begin{equation}\label{wartosc_x(nh+a)}
x(nh+a)=x_a+x_0\left(_0\Delta^{-\alpha}_h 1\right)(nh+a)+\left(_0\Delta^{-\alpha}_hg\right)(nh+a)\,,
\end{equation}
where $g(nh)=\left(_0\Delta^{-\beta}_h\tilde{f}\right)(nh+b)$.

Firstly we prove the formula for the unique solution in linear case, i.e. when  in equation (\ref{uklad:ogolny}): $f(nh, x(nh+a))=A x(nh+a)$, where $A$ is a constant square matrix of degree $n$.

\begin{theorem}
The solution to the system
\begin{subequations}\label{uklad_lin}
\begin{eqnarray}
\left(_a\Delta^{\alpha}_{h,*}x\right)(nh) & = & y(nh+b)\,,\\ \label{uklad3}
\left(_b\Delta^{\beta}_{h,*}y\right)(nh) & = & A x(nh+a) \label{uklad4}
\end{eqnarray}
\end{subequations}
with initial conditions \eqref{warunki}, i.e.  $\left(_a\Delta^{\alpha}_{h,*}x\right)(0)=x_0$ and $x(a)=x_a$,  $x_0, x_a\in \R^n$,  is given by the following:
\begin{equation}\label{formula:LinearSystem}
x(nh+a)=\sum_{k=0}^{+\infty}A^k\left( \varphi_{k,k} x_a+\varphi_{k+1,k}x_0\right)(nh)\,,
\end{equation}
for $n\in\N_0$.
\end{theorem}

\begin{proof}
For $n=0$ let us notice, that only $\varphi_{0,0}(0)=1$, for any $k>0$ the next terms are zero, so
in fact we have: $x(0h+a)=x_a$.

For $n>0$ let us define the following sequence
\begin{equation*}
\begin{split}
x_{m+1}(nh+a) & =x_a\varphi_{0,0}(nh)+x_0\varphi_{1,0}(nh) \\
& +\left( _0\Delta^{-\alpha}_h g_m\right)(nh+a)\,, \ \ \ m\in\N_0\,,
\end{split}
\end{equation*}
where $g_m(nh)=\left(_0\Delta^{-\beta}_h\tilde{f}_m\right)(nh+b)$ and $\tilde{f}_m(nh)=A x_m(nh+a)$ with $x_0(nh+a)=x_a$.

We calculate the first step. As $\tilde{f}_0(nh)=A x_0(nh+a)=A x_a
$, then $g_0(nh)=A x_a \left(_0\Delta^{-\beta}1\right)(nh+b)=A x_a\varphi_{0,1}(nh) 
$.
Going further,
\[x_1(nh+a)=x_a\varphi_{0,0}(nh)+x_0\varphi_{1,0}(nh) +\left( _0\Delta^{-\alpha}_h g_0\right)(nh+a)\,.\]
What  could be written as
\[x_1(nh+a)=\left(x_a\varphi_{0,0}+x_0\varphi_{1,0}+A x_a\varphi_{1,1}\right)(nh)\,.\]
and, using Proposition \ref{fi}, we get
\begin{equation*}
\begin{split}
x_2(nh+a)=\left(x_a\varphi_{0,0}+x_0\varphi_{1,0}+A x_a\varphi_{1,1}+A x_0\varphi_{2,1}+A^2 x_a\varphi_{2,2}\right)(nh)\,.
\end{split}
\end{equation*}
 and the next element of the considered sequence has the following form:
\begin{equation*}
\begin{split}
x_3(nh+a)=& \left(x_a\varphi_{0,0}+x_0\varphi_{1,0}+A x_a\varphi_{1,1}+A x_0\varphi_{2,1}+A^2 x_a\varphi_{2,2}\right.\\
&\left.+A^2x_0\varphi_{3,2}+A^3x_a\varphi_{3,3}\right)(nh)\,.
\end{split}
\end{equation*}
Taking $m$ tending to $+\infty$ we get formula \eqref{formula:LinearSystem} as the solution of \eqref{uklad_lin} with initial conditions \eqref{warunki}.
\end{proof}

\subsection{Semilinear sequential systems}
Firstly we state some technical lemma and notations.

\begin{lemma}\label{lemat3}
Let $u:(h\N)_0\rightarrow\R$ and  $\alpha>0$. Let $\left(_0\Delta^{-k\alpha}_h \gamma \right)(nh+k\alpha h)=\gamma_1(nh+k\alpha h)$ and $\widetilde{\gamma}_1(nh):=\gamma_1(nh+k\alpha h)$ for $k\in\N$. Then for $k\in\N$ we get
\begin{equation}\begin{split}\label{eq1}
\left(_0\Delta^{-\alpha}_h\widetilde{\gamma}_1\right)(t)
 =\left(_0\Delta^{-(k+1)\alpha}_h \gamma \right)(t+k\alpha h)\,,
\end{split}\end{equation}
where $t=nh +\alpha h$.
\end{lemma}
\begin{proof}
First let us consider the case $k=1$. Then from Proposition \ref{sumy} we can write
\[\left(_{\alpha h}\Delta^{-\alpha}_h\left(_0\Delta^{-\alpha}_h \gamma \right) \right)(t)=\left(_0\Delta^{-2\alpha}_h \gamma \right)(t),\]
where $t=nh+2\alpha h$, $n\in\N_0$.

Let
$\gamma_1(nh+\alpha h)=\left( _0\Delta^{-\alpha}_h \gamma \right)(nh+\alpha h)$ and $\widetilde{\gamma}_1(nh):=\gamma_1(nh+\alpha h)$.
Then
\begin{eqnarray*}\begin{split}
\left(_0\Delta_h^{-\alpha} \widetilde{\gamma}_1 \right)&(nh+\alpha h)=
 \\ =& \frac{h}{\Gamma(\alpha)}\sum\limits_{r=0}^n \left(nh+\alpha h -\sigma(rh) \right)_h^{(\alpha-1)}\widetilde{\gamma}_1 (rh)
 \\  =& \frac{h}{\Gamma(\alpha)}\sum\limits_{s=\alpha}^{n+\alpha}\left(nh+2\alpha h-\sigma(sh) \right)_h^{(\alpha-1)}\gamma_1(sh)
 \\  =& \left(\mbox{}_{\alpha h}\Delta^{-\alpha}_h\gamma_1\right)(nh+2\alpha h)
 \\  =& \left( _0\Delta^{-2\alpha}_h \gamma\right)(nh+2\alpha h)\,.
\end{split}\end{eqnarray*}
The equation (\ref{eq1}) for $k>1$ follows inductively.
\end{proof}

Note that
\begin{equation*}\begin{split}\label{kdelta}
\left(_0\Delta^{-k\alpha}_h \gamma\right)&(nh+k\alpha h)\\
 & =\frac{h}{\Gamma(k\alpha)}\sum_{r=0}^{n}\left(nh+k\alpha h -\sigma(rh) \right)_h^{(k\alpha -1)}\gamma(rh)\,.
\end{split}\end{equation*}
Similarly to the procedure presented in the proof of Lemma \ref{lemat3} we can prove that for $k, s\in \N_0$ and $\alpha>0, \beta>0$:

\begin{multline}\label{odwl}
   \left(_0\Delta_h^{-k\alpha -s\beta}\gamma\right)(nh+k\alpha h+s\beta h)
\\  =\frac{h}{\Gamma(k\alpha+s\beta)}\sum_{r=0}^{n}\left(\left(n+k\alpha+s\beta - r-1\right)h\right)_h^{(k\alpha+s\beta -1)}\gamma(rh)
\\  =h^{k\alpha+s\beta}\sum_{r=0}^{n}\frac{\Gamma(n-r+k\alpha+s\beta)}{\Gamma(k\alpha+s\beta)\Gamma(n-r+1)}\gamma(rh)
\\  =\sum_{r=0}^{n}\binom{n-r+k\alpha+s\beta-1}{n-r}h^{k\alpha+s\beta}\gamma(rh) \,.
\end{multline}

Taking $\mu=k\alpha+s\beta$ and
using formula 
 \eqref{fi_def2} we can write \eqref{odwl} shortly in the following way
\begin{equation}\label{formula:function:nonlinear:part}
   \left(_0\Delta_h^{-\mu
   }\gamma\right)(nh+\mu h)=\sum_{r=0}^n\widetilde{\varphi}_{k,s}(nh-rh)\gamma(rh)\,.
\end{equation}

Moreover,
we can also write direct formula for values $\left(_0\Delta^{-\alpha}_hg\right)(nh+a)$ given in \eqref{wartosc_x(nh+a)} for nonlinear problem.  In fact using Definition \ref{sum} of fractional summation, formula (\ref{fi_def2}) of functions $\widetilde{\varphi}_{k,s}$ and Proposition \ref{sumy} we write \eqref{formula:function:nonlinear:part} as follows:
\begin{equation*}
x(nh+a)=x_a+x_0\left(_0\Delta^{-\alpha}_h 1\right)(nh+a)+\sum_{r=0}^{n}\widetilde{\varphi}_{1,1}(nh-\sigma(rh))f(rh,x(rh+a))\,.
\end{equation*}
Using the power rule formula for $\mu=0$ and by Remark \ref{rem} we can write the recursive formula for the solution to nonlinear problem given by
equations (\ref{uklad:ogolny})  and conditions (\ref{warunki}):
\begin{equation}\label{rec}
x(nh+a)=x_a+x_0\varphi_{1,0}(nh)+\sum_{r=0}^{n}\widetilde{\varphi}_{1,1}(nh-\sigma(rh))f(rh,x(rh+a))\,.
\end{equation}
The given formula (\ref{rec}) also works for $n=0$ as $\widetilde{\varphi}_{1,1}(-h)=0$.
Then $x(0h+a)=x_a$.
We can  also check for example the next step:
\begin{equation*}
x(h+a)=x_a+x_0\varphi_{1,0}(h)+\widetilde{\varphi}_{1,1}(0h)f(0,x(a))=x_a+x_0h^{\alpha}+h^{\alpha+\beta}f(0,x(a))\,.
\end{equation*}
For special semilinear case when $f(nh,x(nh+a))=Ax(nh+a)+\gamma(nh)$ we have $f(0,x(a))=Ax(a)+\gamma(0)$. Then
\begin{equation*}
x(h+a)=\left(I+h^{\alpha+\beta}A\right)x_a+h^{\alpha}x_0+h^{\alpha+\beta}\gamma(0)\,.
\end{equation*}

\begin{theorem}
The solution to the system
\begin{subequations}\label{uklad_lin2}
\begin{eqnarray}
\left(_a\Delta^{\alpha}_{h,*}x\right)(nh) & = & y(nh+b)\,,\\ \label{uklad5}
\left(_b\Delta^{\beta}_{h,*}y\right)(nh) & = & A x(nh+a)+\gamma(nh) \label{uklad6}
\end{eqnarray}
\end{subequations}
with initial conditions \eqref{warunki}, i.e.  $\left(_a\Delta^{\alpha}_{h,*}x\right)(0)=x_0$ and $x(a)=x_a$,  $x_0, x_a\in \R^n$ 
is given by 
\begin{multline}\label{form}
x(nh+a)  =\sum_{k=0}^{+\infty}A^k\left( \varphi_{k,k} x_a+\varphi_{k+1,k}x_0\right)(nh) \\  +\sum_{r=0}^{n}\left( \sum_{k=0}^{\infty} A^k\widetilde{\varphi}_{k+1,k+1}(nh-\sigma(rh))\right)\gamma(rh)\,,
\end{multline}
for $n\in \N_0$.
\end{theorem}

\begin{proof}
For $n=0$ let us notice, that only $\varphi_{0,0}(0)=1$, for any $k>0$ the next terms are zero, so
in fact we have: $x(0h+a)=x_a$.

For $n>0$ based on the proof for linear case   we can write the solution formula as follows:
\begin{equation*}
\begin{split}
x(nh+a) & =\sum_{k=0}^{+\infty}A^k\left( \varphi_{k,k} x_a+\varphi_{k+1,k}x_0\right)(nh) \\ & + \sum_{k=0}^{+\infty} A^k \left(_0\Delta_h^{(k+1)(\alpha+\beta)}\gamma\right)(nh-h+(k+1)(\alpha+\beta)h)\,.
\end{split}
\end{equation*}
Then taking into account the formulas (\ref{odwl}) and \eqref{formula:function:nonlinear:part} we get the form (\ref{form}) as the solution of \eqref{uklad_lin2} with initial
conditions \eqref{warunki}.
\end{proof}

\section{Positivity
}\label{sec4}

Based on \cite{Kaczorek2008,Kaczorek2009} we consider the following definitions.
\begin{definition}
The fractional system \eqref{uklad:ogolny} is called {\em positive} fractional system if and only if $x(nh+a)\in\R^n_\geq$ for any initial conditions $x_a,\ x_0\in\R^n_\geq$.
\end{definition}

Let  $x_0,\ x_a\in \R^n_\geq$ and $\mathcal{X}_n:=\{X:\N\rightarrow\R^n\}$.
The operator $T_{x_0,x_a}:\mathcal{X}_n \rightarrow \mathcal{X}_n$ is defined as
\begin{equation*}\label{operator}
\begin{split}
\left(T_{x_0,x_a}X\right)(n)&:=\sum\limits_{k=0}^{+\infty}A^k\left( \varphi_{k,k} x_a+\varphi_{k+1,k}x_0\right)(nh)\\
&+\sum_{r=0}^{n}\left( \sum_{k=0}^{\infty} A^k\widetilde{\varphi}_{k+1,k+1}(nh-\sigma(rh))\right)\gamma(rh)\,,
\end{split}
\end{equation*}
where  $X\in\mathcal{X}_n$, $X(n)=x(nh+a)$ and functions $\varphi_{k,s},\widetilde{\varphi}_{k,s}$ are given by \eqref{fi_def1} and \eqref{fi_def2}, respectively.

The proof of the following proposition is analogous to similar fact in \cite{Kaczorek2008}.

\begin{proposition}\label{prop:1}
Let $x_0,\ x_a\in \R^n_\geq$ and the right hand side of system \eqref{uklad:ogolny}, $f$~fulfils $f(nh, x(nh+a))\geq0$. Then
$T_{x_0,x_a}\left(\R^n_\geq\right)\subset \R^n_\geq$.
\end{proposition}

\begin{definition}
The fractional system \eqref{uklad:ogolny} is called {\em locally positive fractional system} if and only if for any initial conditions $x_a,\ x_0\in\R^n_\geq=\{x\in\mathbb{R}^n: x_i\geq 0, 1\leq i\leq n\}$ there is $\tau\geq1$  such that $x(nh+a)\in\R^n_\geq$ for $n\in [0,\tau]$.
\end{definition}

\begin{proposition}
If the matrix $I+Ah^{\alpha+\beta}$ is positive and $\gamma(nh)\geq0$ for $n\in\N_0$, then for any $x_0,\ x_a\in\R^n_{\geq}$ the fractional system \eqref{uklad_lin2} is locally positive.

\end{proposition}

\begin{proof}
Since $x(h+a)=\sum_{k=0}^1 A^k(\varphi_{k,k}x_a+\varphi_{k+1,k}x_0)(h)=x_a+\varphi_{1,0}(h)x_0+A\varphi_{1,1}(h)x_a=x_a+h^{\alpha}x_0+ Ah^{\alpha+\beta}x_a$, then with our assumptions, we get local positivity.

\end{proof}

\section*{Acknowledgments}

{ \small The work was supported by Bia{\l}ystok University of Technology grant G/WM/3/12.



\begin{thebibliography}{0}

\bibitem{BalAbd}
T. Abdeljawad and D. Baleanu. Fractional differences and integration by parts.{\it Journal of Computational Analysis and Applications}, 13(3):574--582, 2011.

\bibitem{Abd}
T. Abdeljawad. On Riemann and Caputo fractional differences. {\it Computers and Mathematics with Applications}, doi: 10.1016/j.camwa.2011.03.036, 2011.

\bibitem{atici1}
F.~M. Atici and P.~W. Eloe. Initial value problems in discrete fractional calculus.
   In: Proceedings of the American Mathematical Society, S 0002-9939(08)09626-3, 2008.

\bibitem{AticiEloe}
F.~M. Atici and P.~W. Eloe. A Transform Method in Discrete Fractional Calculus.{\it International Journal of Difference Equations}, 2:165--176, 2007.

\bibitem{BasFerTor}
N. R. O. Bastos, R. A. C. Ferreira and D. F. M. Torres, Necessary optimality conditions for fractional difference problems of the calculus of variations.{\it Discrete Contin. Dyn. Syst. }, 29(2):417--437, 2011.

\bibitem{ChenLou}
F. Chen, X. Luo\ and\ Y. Zhou. Existence results for nonlinear fractional difference equation. {\it Advances in Difference Eq.}, article ID 713201, 12 p., doi: 10.1155/2011/713201, 2011.

\bibitem{FerTor}
R. A. C. Ferreira and D. F. M. Torres. Fractional h-difference equations arising
from the calculus of variations. {\it Appl. Anal. Discrete Math.}, 5(1):110--121,
2011.

\bibitem{GirMoz}
E. Girejko and D. Mozyrska. Overview of the fractional difference operators appearing in linear systems theory.
Alexandre Almeida, Luis Castro, Frank-Olme Speck (edts): {\it Advances in Harmonic Analysis and Operator Theory -- The Stefan Samko Anniversary Volume, Operator Theory: Advances and Applications}, Vol. 229, ISBN: 978-3-0348-0515-5 (to appear in 2013), XII, 388 p, Birkhäuser, accepted in 2012.

\bibitem{GirMozCoimbra}
E. Girejko and D. Mozyrska. Semi-linear systems with Caputo type multi-step differences. {\it Submitted to FDA12, Nanjing, China},  to appear.

\bibitem{Holm}
M. T. Holm. {\it The theory of discrete fractional calculus: Development
and application}, University of Nebraska - Lincoln, 2011.

\bibitem{Kaczorek1}
T. Kaczorek. Reachability and controllability to zero
of positive fractional discrete-time systems.  {\it Machine
Intelligence and Robotic Control}, 6(4), 2007.

\bibitem{Kaczorek2008}
T. Kaczorek. Fractional positive continuous-time linear systems and their reachability. {\it Int. J. Appl. Math. Comput. Sci.}, 18(2), 223--228, 2008.

\bibitem{Kaczorek2009}
T. Kaczorek. Reachability of cone fractional continuous-time linear systems. {\it Int. J. Appl. Math.
Comput. Sci.}, 19(1), 89--93, 2009.

\bibitem{MilRos}
K. S. Miller and B. Ross. Fractional difference calculus. {\it In Proceedings of the International Symposium on Univalent Functions, Fractional Calculus and their
Applications}, pages 139--152, Nihon University, K\={o}riyama, Japan, 1988.

\bibitem{DM_EP12}
D. Mozyrska and E. Paw\l uszewicz. Controllability of $h$-difference  linear control systems with  two fractional  orders.
{\it  Submitted to 13th International Carpathian Control Conference ICCC´2012, Slovak Republik,
May 28-31, 2012}, to appear.

\bibitem{ORT}
M. D. Ortigueira. Fractional discrete-time linear systems.
{\it In: Proc. of the IEE-ICASSP 97}, Munich, Germany, IEEE, New York, 3:2241--2244, 1997.

\bibitem{Pod}
I. Podlubny. {\it Fractional Differential Equations}, AP, New York, NY, 1999.

\end{thebibliography}
\end{document}